\documentclass[12pt, reqno, a4paper]{amsart}

\usepackage{amsmath, amssymb, amsfonts, enumerate}
\usepackage{amsthm}
\usepackage{latexsym, xcolor, epsfig}
\usepackage{hyperref, mathrsfs}
\usepackage{esint}

\usepackage[numbers, square]{natbib}

\usepackage{scalerel}

\renewcommand{\Bbb}{\mathbb}

\newenvironment{nalign}{
    \begin{equation}
    \begin{aligned}
}{
    \end{aligned}
    \end{equation}
    \ignorespacesafterend
}

\newcommand{\Z}{\mathbb{Z}}
\newcommand{\R}{\mathbb{R}}

\newcommand{\Q}{\mathbb{Q}}
\usepackage{dsfont}\newcommand{\1}{\mathds{1}}
\newcommand{\nd}{\text{ and }}

\DeclareMathOperator{\sgn}{sgn}

\newcommand{\eps}{\varepsilon} 
\renewcommand{\epsilon}{\varepsilon} 
\renewcommand{\leq}{\leqslant} 
\renewcommand{\geq}{\geqslant}
\renewcommand{\subset}{\subseteq} 

\newtheorem{theorem}{Theorem}[section]
\newtheorem*{theorem*}{Theorem}

\newtheorem{proposition}[theorem]{Proposition}
\newtheorem*{proposition*}{Proposition}
\newtheorem{lemma}[theorem]{Lemma}
\newtheorem*{lemma*}{Lemma}

\newtheorem*{corollary*}{Corollary}

\newtheorem*{example*}{Example}
\newtheorem{claim}{Claim}
\newtheorem*{claim*}{Claim}

\newtheorem*{observation*}{Observation}

\theoremstyle{definition}
\newtheorem{definition}[theorem]{Definition}
\newtheorem*{definition*}{Definition}

\theoremstyle{remark}

\newtheorem*{remark*}{Remark}

\newtheorem*{question*}{Question}

\numberwithin{equation}{section}

\title[Sets with many solutions]{\vspace{-0.7cm}Maximising the number of solutions to a linear equation in a set of integers}
\author{James Aaronson}
\email{james.aaronson@maths.ox.ac.uk}

\date{}
\begin{document}

\begin{abstract}
Given a linear equation of the form $a_1x_1 + a_2x_2 + a_3x_3 = 0$ with integer coefficients $a_i$, we are interested in maximising the number of solutions to this equation in a set $S \subseteq \mathbb{Z}$, for sets $S$ of a given size.

We prove that, for any choice of constants $a_1, a_2$ and $a_3$, the maximum number of solutions is at least $\left(\frac{1}{12} + o(1)\right)|S|^2$. Furthermore, we show that this is optimal, in the following sense. For any $\varepsilon > 0,$ there are choices of $a_1, a_2$ and $a_3,$ for which any large set $S$ of integers has at most $\left(\frac{1}{12} + \varepsilon\right)|S|^2$ solutions.

For equations in $k \geq 3$ variables, we also show an analogous result. Set $\sigma_k = \int_{-\infty}^{\infty} (\frac{\sin \pi x}{\pi x})^k dx.$ Then, for any choice of constants $a_1, \dots, a_k$, there are sets $S$ with at least $\left(\frac{\sigma_k}{k^{k-1}} + o(1)\right)|S|^{k-1}$ solutions to $a_1x_1 + \dots + a_kx_k = 0$. Moreover, there are choices of coefficients $a_1, \dots, a_k$ for which any large set $S$ must have no more than $\left(\frac{\sigma_k}{k^{k-1}} + \varepsilon\right)|S|^{k-1}$ solutions, for any $\varepsilon > 0$.
\end{abstract}

\maketitle

\section{Introduction}\label{sec:intro}

Let $a_1, a_2 \nd a_3$ be fixed coprime integers, none of which is zero. We will consider the linear equation
\begin{equation}\label{eqn:to_solve}
a_1x_1 + a_2x_2 + a_3x_3 = 0. 
\end{equation}
In this paper, we are interested in the problem of finding sets with as many solutions to (\ref{eqn:to_solve}) as possible. This leads to the following definition.
\begin{definition}
Given a finite set $S \subset \Z$, define $T(S) = T_{a_1, a_2, a_3}(S)$ to be the number of triples $x_1, x_2, x_3 \in S$ satisfying (\ref{eqn:to_solve}).
\end{definition}

The trivial upper bound on $T(S)$ is $T(S) \leq |S|^2$. This is because, for any choice of $x_1 \nd x_2$, there is at most one choice of $x_3$ such that $a_1x_1 + a_2x_2 + a_3x_3 = 0$, namely $x_3 = \frac{-a_1x_1-a_2x_2}{a_3}$. We are interested in making $T(S)$ as large as possible, for a fixed size $|S|$. 

For some choices of coefficients $a_1, a_2 \nd a_3$, the exact maximal value of $T(S)$ is known. For example, consider the case $a_1 = a_2 = a_3 = 1$. Then, work of Hardy and Littlewood \cite{HLIneqNotes} and Gabriel \cite{Gabriel} shows that, when $|S|$ is odd, $T(S)$ is maximised when $S$ is an interval centred about 0. This was extended to even $|S|$ by Lev in \cite{lev_max}. In fact, their arguments show that if $S \subseteq \Z$ is a set, and $S'$ is an interval centred about 0 of the same size, then $T_{a_1, a_2, a_3}(S) \leq T_{1,1,1}(S')$. The ideas behind their approaches involve rearrangement inequalites, which are discussed in detail in \cite[Chapter 10]{HLP}, and which inspire some of the arguments in this paper.

Similarly, it is shown by Green and Sisask in \cite[Theorem 1.2]{GreenSisask} and by Lev and Pinchasi in \cite[Theorem 2]{LevPinchasi} respectively that, if $(a_1, a_2, a_3) = (1, -2, \pm 1)$, then $T(S)$ is again maximised when $S$ is an interval centred at 0.

The set of solutions to $x_1 - 2x_2 + x_3 = 0$ is precisely the set of three-term arithmetic progressions; that is, the set of affine shifts of the set $\{0, 1, 2\}$. By analogy with this, Bhattacharya, Ganguly, Shao and Zhao considered longer arithmetic progressions; in \cite[Theorem 2.4]{uppertails}, they proved that the number of $k$ term arithmetic progressions in a set $S$ of $n$ integers is maximised when $S$ is an interval. 

Ganguly asked \cite{pc_ganguly} about other affine patterns; in particular, finding sets $S$ with as many affine copies of $\{0, 1, 3\}$, or solutions to $x + 2y = 3z$, as possible. In this case, such a result would necessarily be less clean; for instance, there are more solutions to $x + 2y = 3z$ in $\{0,1,3\}$ than in $\{0,1,2\}$. 

Indeed, in general, much less is known. For a lower bound on the maximal value of $T(S)$, a fairly good bound is given by the following example.
\begin{proposition}\label{prop:twelfth}
Regardless of the values of $a_1, a_2 \nd a_3$, there are choices of $S$ with $|S|$ arbitrarily large, for which $T(S) \geq \frac{1}{12}|S|^2 + O(|S|)$.
\end{proposition}

\begin{proof}
The idea behind the construction is to split $S$ into three pieces $S_1, S_2 \nd S_3$, of roughly equal size, for which there are many solutions to $a_1x_1 + a_2x_2 + a_3x_3 = 0$ with each $x_i$ taken from $S_i$. Let $M$ be a large integer, which we assume to be divisible by 6. We will define 
\begin{align*}
S_1 &= a_2a_3[-M/6, M/6] \\
S_2 &= a_1a_3[-M/6, M/6] \\
S_3 &= a_1a_2[-M/6, M/6], 
\end{align*}
where $[-M/6, M/6]$ refers to the set of integers with absolute value no greater than $M/6$, and set $S = S_1 \cup S_2 \cup S_3$. Then, $|S|$ is certainly no more than $M$.

However, we may find a large collection of triples $(x_1, x_2, x_3)$ by choosing $x_1 \in S_1 \nd x_2 \in S_2$ arbitrarily, and selecting those for which $x_3 = \frac{-a_1x_1-a_2x_2}{a_3}$ is in $S_3$. If $x_1 = a_2a_3x_1' \nd x_2 = a_1a_3x_2'$, then we have $x_3 = -a_1a_2(x_1' + x_2')$. Therefore, a pair $(x_1', x_2')$ will give rise to a solution precisely when $|x_1' + x_2'| \leq M/6$. 

We may compute the number of such pairs $(x_1', x_2')$ as the sum
$$
\sum_{x_1' = -M/6}^{M/6} M/3 + 1 - |x_1'| = \frac{1}{12}M^2 + O(M).
$$
Thus, the number of triples is at least $\frac{1}{12}|S|^2 + O(|S|)$.
\end{proof}

Given this, it is natural to define the following quantity:
\begin{definition}
Define $\gamma_{a_1,a_2,a_3}$ by
$$
\gamma_{a_1,a_2,a_3} = \limsup_{|S| \rightarrow \infty} \dfrac{T(S)}{|S|^2}
$$
where $S$ runs over subsets of $\Z$.
\end{definition}

Thus, the assertion that
\begin{equation}\label{eqn:gamma_uniform_bound}
\frac{1}{12} \leq \gamma_{a_1, a_2, a_3} \leq \frac{3}{4}
\end{equation}
holds for all $a_1, a_2$ and $a_3$ follows from Proposition \ref{prop:twelfth} and the work of Hardy and Littlewood in \cite{HLIneqNotes}.

As far as the author is aware, exact values for $\gamma_{a_1, a_2, a_3}$ are only known in cases for which $|a_1a_2a_3| \leq 2$ (this includes the cases previously discussed). In particular, we have 
\begin{align}
\gamma_{1,1,\pm 1} &= \frac{3}{4} \label{eqn:trivial_case}\\ 
\gamma_{1,-2,\pm 1} &= \frac{1}{2} \label{eqn:abc_is_2}
\end{align} 
$\gamma_{1,-2,1}$ is \cite[Theorem 1.2]{GreenSisask}, and $\gamma_{1,-2,-1} = \frac{1}{2}$ is \cite[Theorem 2]{LevPinchasi}. The same holds in the third non-equivalent case with $|a_1a_2a_3| = 2$, namely $\gamma_{1, 2, 1} = \frac{1}{2}$. Even the value of $\gamma_{1,2,-3}$ is not known, although the author conjectures that it is $\frac{1}{3}$, which is the value calculated for $S = [-M/2, M/2]$.

The main theorem of this paper is a converse, of sorts, to Proposition \ref{prop:twelfth}. In particular, we will prove the following.

\begin{theorem}\label{thm:main_theorem}
The constant $\frac{1}{12}$ in the statement of Proposition \ref{prop:twelfth} is optimal, in the following sense. For any $\eps > 0$, there exists a choice of $a_1, a_2 \nd a_3$ for which $\gamma_{a_1, a_2, a_3} \leq \frac{1}{12} + \eps$.
\end{theorem}

In view of this theorem, (\ref{eqn:gamma_uniform_bound}) gives the best possible bounds on $\gamma_{a_1, a_2, a_3}$ that are independent of the coefficients $a_i$.

The plan for this paper is as follows. In Section \ref{sec:import_lemmas}, we will record some additive combinatorial lemmas that we will need in order to establish Theorem \ref{thm:main_theorem}. In Section \ref{sec:proof_of_thm}, we will use these lemmas to prove Theorem \ref{thm:main_theorem}.

One might also ask about generalising Theorem \ref{thm:main_theorem} to other settings. For instance, given a system of $m$ linear equations in $k$ variables (where we assume that $m \leq k - 2$), can we prove an analogue of Theorem \ref{thm:main_theorem}?

If $m = 1$, then an analogue of Proposition \ref{prop:twelfth} holds for any value of $k \geq 3$. Set 
\begin{equation}\label{eqn:definition_of_sigma}
\sigma_k = \int_{-\infty}^{\infty} \left(\frac{\sin \pi x}{\pi x} \right)^k dx.
\end{equation}
Then, for any choice of coefficients $a_1, \dots, a_k$, there are sets $S$ with at least $\frac{\sigma_k}{k^{k-1}}|S|^{k-1} + O(|S|^{k-2})$ solutions to $a_1x_1 + \dots + a_kx_k = 0$. We will discuss (\ref{eqn:definition_of_sigma}) further in Section \ref{sec:generalisation_to_k_var}.

Furthermore, the corresponding analogue of Theorem \ref{thm:main_theorem} holds. For any $\eps > 0$, there are choices of coefficients $a_1, \dots, a_k$ for which any large set $S$ must have no more than $\left(\frac{\sigma_k}{k^{k-1}} + \varepsilon\right)|S|^{k-1}$ solutions. For instance, for any small positive $\eps$ we can find coefficients $a_1, a_2, a_3 \nd a_4$ with the property that 
$
T(S) \leq \left( \frac{1}{96} + \eps \right) |S|^3,
$
where $T(S)$ counts the number of solutions to $a_1x_1 + a_2x_2 + a_3x_3 + a_4x_4 = 0.$ We will discuss this in Section \ref{sec:generalisation_to_k_var}.

On the other hand, the opposite is true in the case that $m > 1$. Indeed, it is possible to show that there is \emph{no} constant $c > 0$, such that for any system of 2 equations in 4 variables, there are large sets $S$ with at least $c |S|^2$ solutions to the system. We will prove this fact in Section \ref{sec:generalisation_to_systems}.

\subsection*{Notation}

As we have already noted, $T(S)$ will be the number of solutions to $a_1x_1 + a_2x_2 + a_3x_3$ in S. We can extend this by defining $T(S_1, S_2, S_3)$ to be the number of solutions to $a_1x_1 + a_2x_2 + a_3x_3$, where $x_i \in S_i$.

We will use the notation $a \cdot S$ to denote the set $\{ax, x \in S\}$.

We will also make frequent use of the Vinogradov notation $f \ll g$ to mean that $f = O(g)$. When the $\ll$ is subscripted, we allow the implicit constant to depend on the subscripts.

The author is supported by an EPSRC grant EP/N509711/1. The author would like to thank his supervisor, Ben Green, for his continued support and encouragement, and the anonymous referee for a thorough reading of a previous version of this paper.

This version of the paper replaces a previous version \cite{me}. The argument used to prove Theorem \ref{thm:main_theorem} is replaced with a new argument which avoids appealing to the arithmetic regularity lemma (and can handle a wider class of equations), and the results of Section \ref{sec:generalisation_to_systems} are new to this version.

\section{Additive Combinatorial Lemmas}\label{sec:import_lemmas}

In this section, we will collect some lemmas that will be necessary for the proof of Theorem \ref{thm:main_theorem}.

For any set $A \subseteq \Z$, let $\delta[A]$ be its growth under the differencing operator, $\frac{|A-A|}{|A|}$. If $A \nd B$ are two sets of integers, let the \emph{additive energy} between $A \nd B, E(A, B)$, be defined by
\[
E(A,B) = \# \{(a_1,b_1,a_2,b_2) \in A \times B \times A \times B: a_1 + b_1 = a_2 + b_2\}.
\]
It is easy to see that this satisfies the following inequalities:
\begin{nalign}\label{eqn:bound_energy}
E(A,B) &\leq |A|^2|B| \\
E(A,B) &\leq |A||B|^2 \\
E(A,B) &\leq |A|^{3/2}|B|^{3/2},
\end{nalign}
the third of which follows immediately from the first two.

We will require the following lemma, which states that, when two sets $A \nd B$ have $\delta[A] \nd \delta[B]$ small, and if $E(A, B)$ is large, then $|A - B|$ is also small.
\begin{lemma}[\protect{\cite[Lemma 3.1 (iv)]{GreenSisask}}]\label{lem:gs_doubling_lem}
Suppose that $A, B \subseteq \Z$ are sets with $E(A, B) \geq \eta |A|^{3/2}|B|^{3/2}$.

Then, $|A - B| \leq \frac{\delta[A]\delta[B]}{\eta} |A|^{1/2}|B|^{1/2}$.
\end{lemma}

We will also require a weak form of a structure theorem due to Green and Sisask.
\begin{theorem}[\protect{\cite[Proposition 3.2]{GreenSisask}}]\label{thm:green_sisask_structure}
Let $\eps_1 \in (0,1/2)$ be a parameter. Then there are choices of (large) integers $K_1 = K_1(\eps_1) \nd K_2 = K_2(\eps_1)$ with the following property. For any set $S \subseteq \Z$, there is a decomposition of $S$ as a disjoint union $S_1 \amalg \dots \amalg S_n \amalg S_0$ such that
\begin{enumerate}
\item \textup{(Components are large)} $|S_i| \geq |S|/K_1$ for $i=1,\dots,n$;
\item \textup{(Components are structured)} $\delta[S_i] \leq K_2$ for $i = 1,\dots,n$;
\item \textup{(Noise term)} $E(S_0,S) \leq \eps_1 |S|^3$.
\end{enumerate}

Observe that property (1) guarantees that $n \leq K_1$.
\end{theorem}

The quantity $T(S_1, S_2, S_3)$ is related to the additive energy via the following lemma.
\begin{lemma}\label{lem:easy_TNRG_lemma}
Suppose that $S_1, S_2 \nd S_3 \subseteq \Z$ are finite sets. Then
$$ 
T(S_1, S_2, S_3)^2 \leq E(a_1\cdot S_1, a_2\cdot S_2)|S_3|.
$$
\end{lemma}

\begin{proof}
For any $t \in \Z$, let $\mu(t)$ denote the number of ways of writing $t = a_1x_1 + a_2x_2$, for $x_i \in S_i$. Thus, by definition, 
\[E(a_1\cdot S_1, a_2\cdot S_2) = \sum_{t} \mu(t)^2.\]
Now, we see that 
\begin{equation*}
T(S_1, S_2, S_3)^2 = \left(\sum_{t \in -a_3 \cdot S_3} \mu(t)\right)^2 \leq \left(\sum_t \mu(t)^2\right)|S_3| = E(a_1\cdot S_1, a_2\cdot S_2) |S_3|,
\end{equation*}
the inequality following from Cauchy-Schwarz. This completes the proof of Lemma \ref{lem:easy_TNRG_lemma}.
\end{proof}

The following two facts are standard results in additive combinatorics.
\begin{lemma}[Ruzsa triangle inequality, \protect{\cite[Lemma 2.6]{TaoVu}}]\label{lem:ruzsa_triangle}
For sets $A, B, C \subseteq \Z$,
\[
|A - C| |B| \leq |A - B||B - C|.
\]
\end{lemma}

\begin{lemma}[Energy Cauchy-Schwarz, \protect{\cite[(2.9)]{TaoVu}}]\label{lem:energy_cs}
For sets $A, B \subseteq \Z$,
\[
E(A, B) \leq E(A,A)^{1/2} E(B,B)^{1/2}.
\]
\end{lemma}

We will require the following lemma bounding $T(S_1, S_2, S_3)$.
\begin{lemma}\label{lem:uniform_bound_symmetric}
Suppose that $S_1, S_2, S_3 \subseteq \Z$ are sets with sizes $s_1, s_2 \nd s_3$ respectively. Then, we have the bound
\begin{equation}\label{eqn:uniform_bound}
T(S_1, S_2, S_3) \leq \frac{1}{4}\left(s_1s_2 + s_2s_3 + s_1s_3 + 1\right).
\end{equation}
\end{lemma}

\begin{proof}
We will first prove Lemma \ref{lem:uniform_bound_symmetric} in the case that $a_1, a_2 \nd a_3$ are all 1.

Without loss of generality, assume that $s_1 \leq s_2 \leq s_3$.

Suppose first that $s_3 \geq s_1 + s_2$. In that case, we have
\begin{align*}
T(S_1, S_2, S_3) &\leq s_1s_2 \\
&\leq \frac{1}{4} (s_1 + s_2)^2 \\
&\leq \frac{1}{4} (s_1s_3 + s_2s_3).
\end{align*}
The first line follows from the trivial observation that for each pair of $x \in S_1 \nd y \in S_2$, there can be at most one solution to $x + y + z = 0$ with $z \in S_3$. The third line follows from our assumption on $s_3$. Thus, (\ref{eqn:uniform_bound}) follows in this case.

Now, suppose that $s_3 < s_1 + s_2$. In this case, we may apply \cite[Lemma 2]{LevPinchasi}, which states that
\begin{equation}
T(S_1, S_2, S_3) \leq \frac{2(s_1s_2 + s_2s_3 + s_3s_1) - (s_1^2 + s_2^2 + s_3^2) + 1}{4}.
\end{equation}
(\ref{eqn:uniform_bound}) follows in this case via an easy application of the Cauchy-Schwarz inequality.

Finally, for arbitrary coefficients $a_1, a_2 \nd a_3$, observe that 
\begin{align*}
T_{a_1, a_2, a_3}(S_1, S_2, S_3) &= T_{1,1,1}(a_1 \cdot S_1, a_2 \cdot S_2, a_3 \cdot S_3) \\
&\leq \frac{1}{4}\left(s_1s_2 + s_2s_3 + s_1s_3 + 1\right).
\end{align*}
This completes the proof of Lemma \ref{lem:uniform_bound_symmetric}.
\end{proof}

Finally, we will require the following theorem of Bukh:
\begin{theorem}[\protect{\cite[Theorem 1.2]{Bukh_sod}}]\label{thm:bukh_sod_thm}
Given two coprime integers $\lambda_1 \nd \lambda_2$, we have that for any $S \subseteq \Z$,
\[ 
    |\lambda_1 \cdot S + \lambda_2 \cdot S| \geq (|\lambda_1| + |\lambda_2|)|S| - o_{\lambda_1, \lambda_2}(|S|)
\]
\end{theorem}

\section{Proof of Theorem \ref{thm:main_theorem}}\label{sec:proof_of_thm}

In this section, we will use the lemmas of Section \ref{sec:import_lemmas} to prove Theorem \ref{thm:main_theorem}. We must prove that, given a suitable choice of $a_1, a_2 \nd a_3$, all sufficiently large sets $S$ have $T(S) \leq \left(\frac{1}{12} + \eps \right)|S|^2$.

Let $\eps > 0$. Given our choice of $\eps$, we must choose the values of the coefficients $a_1, a_2 \nd a_3$; we will do so later. Suppose that $S$ is a sufficiently large set. We will immediately apply the structure theorem, Theorem \ref{thm:green_sisask_structure}, to $S$, with $\eps_1 = \left(\frac{\eps}{6}\right)^4$. This gives us a decomposition $S = S_1 \amalg \dots \amalg S_n \amalg S_0$. We will start by showing that the contribution to $T(S)$ from solutions $a_1x_1 + a_2x_2 + a_3x_3 = 0$, with at least one of the $x_i$ taken from $S_0$, is small.

\begin{lemma}\label{lem:ignore_S0}
The number of solutions to $a_1x_1 + a_2x_2 + a_3x_3 = 0$ in $S$, where some $x_i$ is taken from $S_0$, is no greater than $\frac{\eps}{2}|S|^2$.
\end{lemma}

\begin{proof}
The number of such solutions may be upper bounded by 
\[T(S_0, S, S) + T(S, S_0, S) + T(S, S, S_0),\]
and so it suffices to show that each term is no greater than $\frac{\eps}{6}|S|^2$.

Applying Lemmas \ref{lem:easy_TNRG_lemma} and \ref{lem:energy_cs}, we have
\begin{align*}
T(S_0, S, S)^2 &\leq E(a_1 \cdot S_0, a_2 \cdot S) |S| \\
&\leq E(S_0, S_0)^{1/2} E(S, S)^{1/2} |S| \\
&\leq \eps_1^{1/2} |S|^4,
\end{align*}
from which it follows that $T(S_0, S, S) \leq \eps_1^{1/4}|S|^2$. By our choice of $\eps_1$, this gives exactly what we claimed.
\end{proof}

At this point, we must bound the number of solutions to $a_1x_1 + a_2x_2 + a_3x_3 = 0$ where each of $x_1, x_2 \nd x_3$ is taken from an $S_i$ with $i \geq 1$. To do this, we will start by restricting which triples $(i, j, k)$ can have the property that there are many solutions with $x_1 \in S_i, x_2 \in S_j \nd x_3 \in S_k$. For instance, the fact that $\delta[S_1]$ is small, together with an assumption that $a_1$ and $a_2$ are coprime and $|a_1 + a_2|$ is large, will imply that there cannot be too many solutions with $x_1, x_2 \nd x_3$ all in $S_1$. 

In particular, this will give us a fairly rigid structure on the collection of triples $S_i, S_j, S_k$ such that $T(S_i, S_j, S_k)$ can give a non-trivial contribution to $T(S, S, S)$. In order to quantify this structure, we will draw a labelled digraph $G$ whose vertices correspond to the $S_i$ with $i \geq 1$. We will draw an edge from $S_i$ to $S_j$ with label $\frac{a_1}{a_2}$ if and only if $T(S_i, S_j, S) \geq \frac{\eps}{24 K_1^2} |S|^2$, where $K_1$ is as in the statement of Theorem \ref{thm:green_sisask_structure}. Similarly, we will draw an edge with label $\frac{a_3}{a_2}$ if $T(S, S_j, S_i) \geq \frac{\eps}{24 K_1^2} |S|^2$, and similarly for the other four possible labels.

In particular, observe that if there is an edge from $S_i$ to $S_j$ with label $x$, then there will be an edge from $S_j$ to $S_i$ with label $x^{-1}$. Our definition of $G$ does not necessarily preclude the existence of multiple edges between $S_i$ and $S_j$ (with different labels), or edges from $S_i$ to $S_i$. However, as part of the proof, we will show that this cannot happen, provided that we assume a suitable hypothesis on $a_1, a_2 \nd a_3$.

First, we will show that $G$ captures almost all of the solutions to $a_1x_1 + a_2x_2 + a_3x_3 = 0$.
\begin{lemma}\label{lem:ignore_bad_triples}
Say that a triple $S_i, S_j \nd S_k$ is \emph{good} if and only if the six relevant edges are present. For example, $S_i \rightarrow S_j$ has label $\frac{a_1}{a_2}$, $S_k \rightarrow S_j$ has label $\frac{a_3}{a_2}$, and so on. Say that a triple is \emph{bad} otherwise.

Then, the total number of solutions to $a_1x_1 + a_2x_2 + a_3x_3 = 0$ among all of the bad triples is at most $\frac{\eps}{4} |S|^2$.
\end{lemma}

\begin{proof}
There are six ways a triple $(S_i, S_j, S_k)$ can be bad. One such way is if there is no edge from $S_i$ to $S_j$ with label $\frac{a_1}{a_2}$.

Let us count the total number of solutions among triples for which the $\frac{a_1}{a_2}$ edge is missing. That is
\begin{align*}
\sum_{\text{such } S_i, S_j} T(S_i, S_j, S) &\leq K_1^2 \frac{\eps}{24 K_1^2} |S|^2 \\
&= \frac{\eps}{24} |S|^2,
\end{align*}
since the number of pairs $S_i, S_j$ is bounded by $K_1^2$.

Summing this over the six possible ways for a triple to be bad completes the proof of Lemma \ref{lem:ignore_bad_triples}.
\end{proof}

In view of Lemmas \ref{lem:ignore_S0} and \ref{lem:ignore_bad_triples}, it remains to show that the number of solutions among the good triples is at most $\left(\frac{1}{12} + \frac{\eps}{4}\right)|S|^2$, for a suitable choice of the coefficients $a_1, a_2 \nd a_3$. The values we will choose are $a_1 = 1, a_2 = M \nd a_3 = M+1$, where
\begin{equation}\label{eqn:condition_on_M}
M > \left(\frac{1000 K_1^4K_2^2}{\eps^2}\right)^{K_1}.
\end{equation}

We can now prove the following lemma:
\begin{lemma}\label{lem:no_bad_cycles}
With the values of $a_1, a_2 \nd a_3$ that we have chosen, the product of the labels along any cycle in $G$ must be 1.
\end{lemma}

\begin{remark*}
This immediately tells us that $G$ has no loops (edges from a vertex to itself). In view of the fact that an edge from $S_i$ to $S_j$ with label $x$ is accompanied by an edge from $S_j$ to $S_i$ with label $x^{-1}$, this also tells us that there can be at most one edge from $S_i$ to $S_j$.
\end{remark*}

\begin{remark*}
We have chosen particular values of the $a_i$ for simplicity; indeed, we only need a single choice of coefficients to work in order to establish Theorem \ref{thm:main_theorem}. However, the same argument is able to establish Lemma \ref{lem:no_bad_cycles}, and thus also Theorem \ref{thm:main_theorem}, for a much wider class of equations. For example, whenever $a_1, a_2 \nd a_3$ are coprime, and at least two of the three coefficients are large enough, then the analogue of Lemma \ref{lem:no_bad_cycles} holds, and thus $\gamma_{a_1, a_2, a_3} < 1/12 + \eps$.

Conversely, it does not suffice for just one of the $a_i$ to be large. For example, if $a_1 = a_2 = 1$, then it can be shown that, for $S$ a slightly modified version of the set in Proposition \ref{prop:twelfth}, $T_{1, 1, a_3}(S) > \frac{1}{5} |S|^2$ for any $a_3$.
\end{remark*}

\begin{proof}[Proof of Lemma \ref{lem:no_bad_cycles}]
Suppose there is a cycle whose label product is not 1; consider a shortest such cycle. By minimality, such a cycle may have no repeated vertices, and thus must have at most $K_1$ vertices. Thus, without loss of generality the cycle is $S_1, S_2, \dots, S_k, S_1$, where $S_i \rightarrow S_{i+1}$ has label $t_i$ (with $S_{k+1} = S_1$), and $k \leq K_1$.

By Lemma \ref{lem:easy_TNRG_lemma}, we deduce that for each $i$, 
\[
E(t_i \cdot S_i, S_{i+1}) \geq \frac{\eps^2}{576 K_1^4}|S|^3.
\]

Now, let us apply Lemma \ref{lem:gs_doubling_lem} to $S_i \nd S_{i+1}$. We have that
\[
E(t_i \cdot S_i, S_{i+1}) \geq \frac{\eps^2}{576 K_1^4} |S_i|^{3/2}|S_{i+1}|^{3/2}, 
\]
and so we deduce that
\begin{align}
|t_i \cdot S_i - S_{i+1}| &\leq \delta[S_i] \delta[S_{i+1}] |S_i|^{1/2}|S_{i+1}|^{1/2} \frac{576K_1^4}{\eps^2} \nonumber\\
&\leq \frac{576 K_1^4K_2^2}{\eps^2}|S_i|^{1/2}|S_{i+1}|^{1/2}.
\end{align}

Now, we can prove, by inductively applying Lemma \ref{lem:ruzsa_triangle}, that
\begin{equation}
|t_1t_2 \dots t_i \cdot S_1 - S_{i+1}| \leq \left(\frac{576 K_1^4K_2^2}{\eps^2}\right)^i |S_1|^{1/2}|S_{i+1}|^{1/2}.
\end{equation}
Thus, setting $i=k$, we learn that
\begin{equation}
|t_1t_2 \dots t_k \cdot S_1 - S_1| \leq \left(\frac{576 K_1^4K_2^2}{\eps^2}\right)^{K_1} |S_1|,
\end{equation}
since $k \leq K_1$. 

By hypothesis, $t_1t_2 \dots t_k \neq 1$. However, we know that $t_1t_2\dots t_k$ can be written in the form $M^{e_1}(M+1)^{e_2}$ for some integers $e_i$ not both zero. Suppose that $e_1$ is nonzero; the argument is similar if $e_2$ is nonzero.

Write $t_1t_2 \dots t_k = \frac{r}{s}$ for coprime integers $r \nd s$; our hypothesis tells us that $M$ must divide $r$ or $s$. Therefore, 
\[
|r| + |s| > \left(\frac{1000 K_1^4K_2^2}{\eps^2}\right)^{K_1},
\]
as a consequence of (\ref{eqn:condition_on_M}).

Thus, we have shown that $|r \cdot S_1 - s \cdot S_1| \leq \left(\frac{576 K_1^4K_2^2}{\eps^2}\right)^{K_1} |S_1|$. But, if $S_1$ is sufficiently large, this contradicts Theorem \ref{thm:bukh_sod_thm}, which states that
\[
|r \cdot S_1 - s \cdot S_1| > \left(\frac{1000 K_1^4K_2^2}{\eps^2}\right)^{K_1} |S_1|,
\]
whenever $|S_1|$ is sufficiently large.

This contradiction completes the proof of Lemma \ref{lem:no_bad_cycles}.
\end{proof}

To complete the proof of Theorem \ref{thm:main_theorem}, we just need to bound the number of solutions to $a_1x_1 + a_2x_2 + a_3x_3 = 0$, with $x_1, x_2, x_3$ taken from a good triple. The following lemma will achieve this.
\begin{lemma}\label{lem:bounding_good_solutions}
Suppose we choose $a_1 = 1, a_2 = M \nd a_3 = M+1$, as in the statement of Lemma \ref{lem:no_bad_cycles}.

Then the number of solutions to $a_1x_1 + a_2x_2 + a_3x_3 = 0$ taken from good triples is bounded above by $\left(\frac{1}{12} + \frac{\eps}{4}\right)|S|^2$, whenever $|S|$ is large enough.
\end{lemma}

\begin{proof}
We will start by defining a function 
\begin{equation}\label{eqn:defn_of_depth}
d : [n] \rightarrow \Q,
\end{equation}
with the property that if $S_i \rightarrow S_j$ has label $t$, then $d(j) = td(i)$.

One way we can do this is as follows. For each connected component $G'$ of $G$, choose the smallest value of $i$ such that $S_i$ is in $G'$, and set $d(i) = 1$. Then, for any other $j$ with $S_j$ in $G'$, $d(j)$ is determined by the product of the labels on any path from $S_i$ to $S_j$. Lemma \ref{lem:no_bad_cycles} guarantees that this value does not depend on the path chosen.

Now, for each $d$, let $R_d = \cup_{i : d(i) = d} S_i$. Suppose that $S_i, S_j, S_k$ is a good triple, in that order (so, for example, the label on $S_i \rightarrow S_j$ is $a_2/a_1$). Then, setting $d$ = $a_1d(i)$, we have that $d(i) = d/a_1, d(j) = d/a_2 \nd d(k) = d/a_3$.

Therefore, all of the solutions coming from the good triple $S_i, S_j, S_k$ will be counted in $T(R_{d/a_1}, R_{d/a_2}, R_{d/a_3})$, and so an upper bound for the total number of solutions coming from good triples is
\[
\sum_{d} T(R_{d/a_1}, R_{d/a_2}, R_{d/a_3}),
\]
where the sum is taken over all $d$ such that all three of the $R_i$ exist (in particular, there can be no more than $n$ terms in the sum).

We may apply Lemma \ref{lem:uniform_bound_symmetric} to give an upper bound for this.
\begin{align}
\sum_{d} T(R_{d/a_1}, R_{d/a_2}, R_{d/a_3}) &\leq \frac{1}{4}\sum_d |R_{d/a_1}| |R_{d/a_2}| + |R_{d/a_1}| |R_{d/a_3}| + |R_{d/a_2}| |R_{d/a_3}| + 1 \nonumber\\
&\leq \frac{1}{4}\sum_{d_1 \sim d_2} |R_{d_1}||R_{d_2}| + K_1, \label{eqn:sum_of_RR}
\end{align}
where the sum on the second line is over \emph{unordered} pairs $d_1, d_2$ such that $d_1/d_2$ is equal to the ratio between two of the $a_i$. The second inequality follows because if $d_1 \sim d_2$, then there is exactly one ratio $a_i/a_j$ such that $d_1/d_2 = a_i/a_j$. Thus, the term  $|R_{d_1}||R_{d_2}|$ appears in at most one of the sums on the right hand side of the first line.

Finally, for $i = 0, 1 \nd 2$, define the quantity $X_i$ by
\[
X_i = \sum_{\substack{d = M^{e_1}(M+1)^{e_2} \\ e_1 - e_2 = i \mod 3}} |R_d|.
\]
By our construction of $d$, each $|R_d|$ appears as a term in exactly one of the $X_i$. Furthermore, $d_1 \sim d_2$ only if $R_{d_1}$ and $R_{d_2}$ are in different sums $X_i$, and any term $|R_{d_1}||R_{d_2}|$ appears at most once in (\ref{eqn:sum_of_RR}). Consequently, we have the upper bound
\begin{align*}
\frac{1}{4}\sum_{d_1 \sim d_2} |R_{d_1}||R_{d_2}| &\leq \frac{1}{4}(X_0X_1 + X_0X_2 + X_1X_2) \\
&\leq \left(\frac{1}{12} + \frac{\eps}{4}\right)|S|^2,
\end{align*}
the latter inequality following from an easy application of Cauchy-Schwarz, since $X_0 + X_1 + X_2 \leq |S|$. This completes the proof of Lemma \ref{lem:bounding_good_solutions}.
\end{proof}

We have now essentially proven Theorem \ref{thm:main_theorem}. Indeed, any solution to $a_1x_1 + a_2x_2 + a_3x_3 = 0$ must either have some $x_i$ in $S_0$, or must come from a bad triple, or must come from a good triple. Combining Lemmas \ref{lem:ignore_S0}, \ref{lem:ignore_bad_triples} and \ref{lem:bounding_good_solutions} gives the result if $|S|$ is large enough. \qed

\section{Equations in more than 3 variables}\label{sec:generalisation_to_k_var}

A fairly natural extension of Theorem \ref{thm:main_theorem} is to ask if a similar result holds for $k$-variable equations
\begin{equation}\label{eqn:to_solve_k_var}
a_1x_1 + \dots + a_kx_k = 0.
\end{equation}
As before, let $T(S)$ be the number of solutions to (\ref{eqn:to_solve_k_var}) in $S$. Similarly, let $T(S_1, \dots, S_k)$ denote the number of solutions with $x_i$ taken from $S_i$. We have a trivial upper bound for $T(S)$, namely that $T(S) \leq |S|^{k-1}$.

Before presenting our analogous example to Proposition \ref{prop:twelfth}, we require some notation and definitions. Let $I_x : \R \rightarrow \R$ denote the indicator function of a (real) interval of length $x$ centred at the origin, so $I_x(y) = 1$ if and only if $|y| \leq \frac{x}{2}$, and $I_x(y) = 0$ otherwise. 
\begin{definition}\label{def:defn_of_sigma}
For an integer $k \geq 3$, define
\begin{equation}\label{eqn:interval_formula_for_sigma}
\sigma_k = (\underbrace{I_1 \ast \dots \ast I_1}_k)(0).
\end{equation}
\end{definition}

\begin{remark*}
In the introduction, we gave the following formula for $\sigma_k$:
\begin{equation}\tag{\ref{eqn:definition_of_sigma}}
\sigma_k = \int_{-\infty}^{\infty} \left(\frac{\sin \pi x}{\pi x} \right)^k dx.
\end{equation}
The equivalence of these forms follows from taking a Fourier transform and applying the convolution identity; the details can be seen in \cite{BorweinIntegral}.

See also \cite{nathanson}, where it can be shown that $\sigma_{2h}$ is the leading coefficient of the polynomial $\Psi_{h}(n)$.
\end{remark*}

\begin{remark*}
$\sigma_k$ obeys a simple asymptotic (see for example \cite{Polya4Asymp}, or \cite{GoddardAsymptotic} for more terms):
\[
\sigma_k = \sqrt{\frac{6}{k \pi}} \left(1 + O(1/k)\right)
\]
as $k \rightarrow \infty$.

We may interpret $\sigma_k$ combinatorially. If $f_k$ is the probability density function of a sum of $k$ independent random variables distributed uniformly on $[-1/2, 1/2]$, then $\sigma_k = f_k(0)$. Thus, the form of the asymptotic for $\sigma_k$ is not surprising, in view of the Central Limit theorem.
\end{remark*}

\begin{definition}\label{def:definition_of_phi}
For $t_1, \dots, t_k$ positive real numbers, define the function $\Phi = \Phi_k$ by
\begin{equation}\label{eqn:definition_of_phi}
\Phi(t_1, \dots, t_k) = (I_{t_1} \ast \dots \ast I_{t_k})(0).
\end{equation}

In particular, $\sigma_k = \Phi_k(1, \dots, 1)$.
\end{definition}

\begin{remark*}
There is an explicit formula for $\Phi$. In general, we have
\begin{equation}\label{eqn:phi_closed_form_k_var}
\Phi_k(t_1, \dots, t_k) = \frac{1}{(k-1)!2^k}\sum_{\eps \in \{\pm 1\}^{k}} \omega (\eps) (\eps \cdot \mathbf{t})^{k-1} \sgn(\eps \cdot \mathbf{t}),
\end{equation}
where $\omega(\eps) = \prod_i \eps_i$ and $\eps \cdot \mathbf{t} = \sum_i \eps_it_i$ and $\sgn$ denotes the sign function. This is established in \cite{BorweinIntegral}.

For $k = 3$, we can write (for $t_1 \leq t_2 \leq t_3$)
\begin{equation}\label{eqn:phi_closed_form_3_var}
\Phi_3(t_1,t_2,t_3) = \begin{cases}
t_1t_2 & t_3\geq t_1+t_2\\
\dfrac{2(t_1t_2+t_2t_3+t_1t_3) - (t_1^2+t_2^2+t_3^2)}{4} & \text{ otherwise}.
\end{cases}
\end{equation}
\end{remark*}

In analogy with Proposition \ref{prop:twelfth}, we have the following.
\begin{proposition}\label{prop:twelfth_k_var}
Let $k \geq 3$ be an integer. For any equation $a_1x_1 + \dots + a_kx_k = 0,$ there are large sets $S$ for which 
\begin{equation}\label{eqn:size_of_S_k_var}
T(S) \geq \frac{\sigma_k}{k^{k-1}} |S|^{k-1} + O(|S|^{k-2}).
\end{equation}
\end{proposition}

The proof of Proposition \ref{prop:twelfth_k_var} will rely on the following fact, which states that, when the coefficients $a_i$ are all 1, long progressions behave somewhat like real intervals.

\begin{proposition}\label{prop:intervals_like_progressions}
Suppose $S_1, \dots, S_k$ are arithmetic progressions centred at the origin, with the same common difference. Let $s_i$ be the number of terms in $S_i$.

Then, the number of solutions to $x_1 + \dots + x_k = 0$ where each $x_i \in S_i$ is
\[
\Phi(s_1, \dots, s_k) + O_k((s_1 + \dots + s_k)^{k-2}).
\]
\end{proposition}
\begin{proof}
We may assume without loss of generality that the progressions $S_i$ have common difference 1. To prove Proposition \ref{prop:intervals_like_progressions}, it suffices to use the following observation.

Suppose that $y_1, \dots, y_{k-1}$ are elements of the real intervals $I_{s_1}, \dots, I_{s_{k-1}}$. Then, we have the following two implications for $k-1$-tuples of real numbers $y_1, \dots, y_{k-1}$.
\begin{itemize}
\item If $|y_1 + \dots + y_{k-1}| \leq \frac{s_k}{2},$ then  $\left|\lfloor y_1\rfloor + \dots + \lfloor y_{k-1}\rfloor \right| \leq \frac{s_k}{2} + k;$

\item If $\left|\lfloor y_1\rfloor + \dots + \lfloor y_{k-1}\rfloor \right| \leq \frac{s_k}{2} - k,$ then $|y_1 + \dots + y_{k-1}| \leq \frac{s_k}{2}.$
\end{itemize}

Now, $T(S_1, \dots, S_k)$ counts the number of $k-1$-tuples of integers $(x_i)_{i=1}^{k-1}$ with $x_i \in S_i$, such that $- \sum_i x_i \in S_k$.

Up to an error which is at most $O_k((s_1 + \dots + s_k)^{k-2})$, this can be written as an integral
\[
\int_{-s_1/2}^{s_1/2} \dots \int_{-s_{k-1}/2}^{s_{k-1}/2} \1_{|\lfloor y_1\rfloor + \dots + \lfloor y_{k-1}\rfloor | \leq s_k/2}dy_1 \dots dy_{k-1}.
\]
The two implications above allow us to show that, up to acceptable error, this is equal to 
\[
\int_{-s_1/2}^{s_1/2} \dots \int_{-s_{k-1}/2}^{s_{k-1}/2} \1_{|y_1+ \dots + y_{k-1} | \leq s_k/2}dy_1 \dots dy_{k-1},
\]
which is equal to $\Phi(s_1, \dots, s_k)$; we omit the details.

\end{proof}

We are now ready to prove Proposition \ref{prop:twelfth_k_var}.
\begin{proof}[Proof of Proposition \ref{prop:twelfth_k_var}]
As in Proposition \ref{prop:twelfth}, we will consider $S$ as the union of $k$ sets $S_1, \dots, S_k$, with the property that $T(S_1, \dots, S_k)$ is large.

The way we will do this is as follows. Let $M$ be a large integer, which we assume to be divisible by $2k$. Define
\[
S_i = \frac{1}{a_i}[-M/2k, M/2k]
\]
for each $i$ with $1 \leq i \leq k$, where we may normalise the sets to consist of integers by multiplying by $\prod_i a_i$. Then, let $S = \cup_i S_i$, so that $|S| \leq M$.

It remains to show that $T(S_1, \dots, S_k) \geq \frac{\sigma_k}{k^{k-1}} M^{k-1} + O(M^{k-2}).$ But this follows as an easy consequence of Proposition \ref{prop:intervals_like_progressions}. Indeed, 
\begin{align*}
T(S_1, \dots, S_k) &= \Phi_k(M/k, \dots, M/k) + O(M^{k-2})\\
&= \left(\frac{M}{k}\right)^{k-1}\Phi_k(1, \dots, 1) + O(M^{k-2}) \\
&= \frac{\sigma_k}{k^{k-1}} M^{k-1} + O(M^{k-2}).
\end{align*}
\end{proof}

Perhaps unsurprisingly, Theorem \ref{thm:main_theorem} also generalises to this setting.
\begin{theorem}\label{thm:main_theorem_k_var}
Let $\eps > 0$. Then, there exist coefficients $a_1, \dots, a_k$ with the property that, for any suitably large set $S$,
\[
T(S) \leq \left( \frac{\sigma_k}{k^{k-1}} + \eps \right)|S|^{k-1} + o(|S|^{k-1}).
\]
\end{theorem}

The proof of Theorem \ref{thm:main_theorem_k_var} is broadly similar to the proof of Theorem \ref{thm:main_theorem}. There are two main places in which the argument slightly differs. Firstly, we must generalise Lemma \ref{lem:easy_TNRG_lemma} to give a bound for $T(S_1, \dots, S_k)$ in terms of $E(S_1, S_2)$:
\begin{lemma}\label{lem:easy_TNRG_lemma_k_var}
Suppose that $S_1, \dots, S_k \subseteq \Z$ are finite sets. Then
\begin{equation*}
T(S_1, \dots, S_k)^2 \leq E(a_1 \cdot S_1, a_2 \cdot S_2) \left(|S_3||S_4|\dots |S_k|\right)^{2 - 1/(k-2)}.
\end{equation*}
\end{lemma}
\begin{proof}
For any $t \in \Z$, let $\mu(t)$ denote the number of ways of writing $t = a_1x_1 + a_2x_2$, for $x_i \in S_i$. Thus, by definition, \[E(a_1\cdot S_1, a_2\cdot S_2) = \sum_t \mu(t)^2.\]
Define $\nu(t)$ to be the number of ways of writing $t = -a_3x_3 - \dots - a_kx_k$, for $x_i \in S_i$. Thus, we see that 
\begin{align*}
T(S_1, \dots, S_k)^2 &= \left(\sum_t \mu(t) \nu(t)\right)^2 \\
&\leq \left(\sum_t \mu(t)^2\right)\left(\sum_t \nu(t)^2\right) \\
&= E(a_1\cdot S_1, a_2\cdot S_2)\left(\sum_t \nu(t)^2\right).
\end{align*}
Finally, we observe that $\sum_t \nu(t)^2$ represents the number of solutions to the equation $$a_3x_3 + \dots + a_kx_k = a_3x_3' + \dots + a_kx_k',$$ and so we can bound it by $(|S_3| \dots |S_k|)^{2 - \frac{1}{k-2}},$ by the same argument used in (\ref{eqn:bound_energy}) to bound the energy.
\end{proof}

Secondly, we will have to apply a $k$ variable analogue of Lemma \ref{lem:uniform_bound_symmetric}. The analogue of this is the following:
\begin{lemma}\label{lem:uniform_bound_symmetric_k_var}
Suppose that $S_1, \dots, S_k \subseteq \Z$ are sets with $|S_i| = s_i$. Then
\[
T(S_1, \dots, S_k) \leq \frac{\sigma_k}{k} \left( \sum_i \hat{s_i} \right) + O\left(\sum_i s_i \right)^{k-2},
\]
where $\hat{s_i} = \prod_{j \neq i} s_j$.
\end{lemma}

\begin{remark*}
This lemma is actually weaker than Lemma \ref{lem:uniform_bound_symmetric}, where the error term was $O(1)$. The weaker error term here comes from our reduction to the real case using Proposition \ref{prop:intervals_like_progressions}; an inductive proof would likely give an $O(\sum s_i)^{k-3}$ error term. However, the $O(\sum s_i)^{k-2}$ error term is sufficient for our purpose.
\end{remark*}

\begin{remark*}
If $k$ is even, we can actually deduce a stronger version of (\ref{eqn:to_prove_about_convolution}) by using H\"older's inequality. We have 
\begin{align*}
\Phi_k(s_1, \dots, s_k) &= \frac{1}{2\pi} \int_{-\infty}^{\infty} \widehat{I_{s_1}}(r) \dots \widehat{I_{s_k}}(r) dr \\
&\leq \frac{1}{2\pi}\left( \prod_i \int_{-\infty}^{\infty} \widehat{I_{s_i}}(r) dr \right)^{1/k} \\
&= \prod_i \Phi_k(s_i, \dots, s_i)^{1/k} \\
&= \sigma_k (s_1 \dots s_k)^{1 - 1/k},
\end{align*}
where the second line used H\"older's inequality along with the fact that $k$ is even. This is stronger than (\ref{eqn:to_prove_about_convolution}) via an application of the AM-GM inequality. 

It is unclear whether the stronger version holds in the case that $k$ is odd; indeed, it is not too hard to establish for $k = 3$ by using (\ref{eqn:phi_closed_form_3_var}). However, this stronger form is not necessary, so we only prove the version we need.
\end{remark*}

\begin{proof}[Proof of Lemma \ref{lem:uniform_bound_symmetric_k_var}]

First, observe that the statement of the lemma is unchanged if we assume without loss of generality that each $a_i$ is 1, since we may replace $S_i$ with $a_i \cdot S_i$.

The first step in the proof is to apply \cite[Theorem 1]{lev_max}, which says that we may take each $S_i$ to be an interval of length $s_i$, roughly centred at the origin (depending on the parity of $s_i$), in order to maximise $T(S_1, \dots, S_k)$. We may immediately apply Proposition \ref{prop:intervals_like_progressions}, which says that
\[
T(S_1, \dots, S_k) = \Phi_k(s_1, \dots, s_k) + O((s_1 + \dots + s_k)^{k-2}).
\]
Thus, it suffices to prove that
\begin{equation}\label{eqn:to_prove_about_convolution}
\Phi(s_1, \dots, s_k) \leq \frac{\sigma_k}{k} \left( \sum_i \hat{s_i} \right).
\end{equation}
This will follow if we can prove that, for positive real numbers $t_1, \dots, t_k$,
\begin{equation}\label{eqn:to_prove_about_phi_flipped}
t_1\dots t_k \Phi(t^{-1}_1, \dots, t^{-1}_k) \leq \frac{\sigma_k}{k}(t_1 + \dots + t_k).
\end{equation}

To prove (\ref{eqn:to_prove_about_phi_flipped}), first observe that equality holds in the case that all of the $t_i$ are equal. Indeed, when $t_i = 1$ the relation follows from the definition of $\sigma$, and for other constant values of $t_i$ the equality follows by homogeneity.

Set $\Theta(t_1, \dots, t_k) = t_1 \dots t_k \Phi(t^{-1}_1, \dots, t^{-1}_k)$. To prove that $\Theta(t_1, \dots, t_k)$ achieves its maximum value (with $t_1 + \dots + t_k$ fixed) when all of the $t_i$ are equal, observe that it will suffice to prove the following claim.
\begin{claim}\label{clm:make_si_equal}
If $t_1 + t_2$ is fixed (as well as each of $t_3, \dots, t_k$), then $\Theta(t_1, \dots, t_k)$ achieves its maximum when $t_1 = t_2$. 
\end{claim}
To see that this claim is sufficient, observe that we may repeatedly replace the largest and smallest of the $t_i$ with their average. In doing so, $\max t_i - \min t_i$ will tend to 0, and we can use the continuity of $\Theta$ to obtain the result.

To prove Claim \ref{clm:make_si_equal}, recall the expression for $\Theta(t_1, \dots, t_k)$:
\begin{align*}
\Theta(t_1, \dots, t_k) &= t_1t_2 (t_3\dots t_k) (I_{t^{-1}_1} \ast I_{t^{-1}_2}) \ast (I_{t^{-1}_3} \ast \dots \ast I_{t^{-1}_k})(0) \\
&= (t_1 t_2 (I_{t^{-1}_1} \ast I_{t^{-1}_2}) \ast g)(0),
\end{align*}
where $g(x) = t_3 \dots t_k (I_{t^{-1}_3} \ast \dots \ast I_{t^{-1}_k})(x)$.

Now, observe that $g$ may be written as a combination of intervals, in the following sense:
\[
g(x) = \int_0^{\infty} h(r) I_r(x) dr,
\]
for some function $h : \R_{>0} \rightarrow \R_{>0}$ with bounded support. (The exception is when $k = 3$, in which case $g$ is just a single interval. But that will not affect the remainder of the proof of Claim \ref{clm:make_si_equal}.)

To see why this is the case, we may use induction. If $k = 4$, then suppose without loss of generality that $t_3 \leq t_4$. Then, we take $h(r) = t_3t_4$ if $\frac{t_3^{-1} - t_4^{-1}}{2} \leq r \leq \frac{t_3^{-1} + t_4^{-1}}{2}$, and 0 otherwise. For $k > 4$, it is easiest to apply the induction hypothesis to $I_{t^{-1}_3} \ast \dots \ast I_{t^{-1}_{k-1}}$, and then use a similar decomposition to the one we used for the $k=4$ case. We omit the details.

In view of this decomposition, proving Claim \ref{clm:make_si_equal} may be reduced to the following claim:
\begin{claim}\label{clm:make_si_equal_against_indicator}
Fix $t_1 + t_2$. Then, for any choice of $t$, we have that $t_1 t_2 (I_{t^{-1}_1} \ast I_{t^{-1}_2} \ast I_t)(0)$ is maximised when $t_1 = t_2$.
\end{claim}

In fact, the easiest way to prove Claim \ref{clm:make_si_equal_against_indicator} is via the following explicit formula for $(I_a \ast I_b \ast I_c)(0)$:
\begin{equation}\label{eqn:formula_for_phi_3}
(I_a \ast I_b \ast I_c)(0) = \begin{cases}
ab & c\geq a+b\\
\dfrac{2(ab+bc+ca) - (a^2+b^2+c^2)}{4} & \text{ otherwise},
\end{cases}
\end{equation}
assuming that $c \geq a,b$ without loss of generality.

Given (\ref{eqn:formula_for_phi_3}), we can prove that $\Theta(a, b, c)$ is a concave function. If, for instance, $c^{-1} > a^{-1} + b^{-1}$, then $\Theta(a, b, c) = c$ which is clearly concave. When $a^{-1}, b^{-1} \nd c^{-1}$ satisfy the triangle inequality, then 
\[
\Theta(a, b, c) = \dfrac{2(a + b + c) - (abc^{-1} + ab^{-1}c + a^{-1}bc)}{4}.
\] We may prove that this is concave by computing the Hessian matrix and showing that it is nonpositive-definite everywhere; for instance, by using Sylvester's Rule. We omit the details.

In particular, $t \Theta(t_1, t_2, t^{-1}) = t_1 t_2 (I_{t^{-1}_1} \ast I_{t^{-1}_2} \ast I_t)(0)$ is concave as a function of $t_1 \nd t_2$. Therefore, 
\[
\frac{1}{2}\left(t \Theta(t_1, t_2, t^{-1}) + t \Theta(t_2, t_1, t^{-1})\right) \leq t \Theta\left(\frac{t_1+t_2}{2}, \frac{t_1+t_2}{2}, t^{-1}\right),
\]
which is exactly the statement of Claim \ref{clm:make_si_equal_against_indicator}. This completes the proof of Claim \ref{clm:make_si_equal}, and thus Lemma \ref{lem:uniform_bound_symmetric_k_var}.
\end{proof}

Armed with our more general Lemmas \ref{lem:easy_TNRG_lemma_k_var} and \ref{lem:uniform_bound_symmetric_k_var}, we may use an argument similar to the proof of Theorem \ref{thm:main_theorem} in section 3 in order to prove Theorem \ref{thm:main_theorem_k_var}.

\begin{proof}[Sketch proof of Theorem \ref{thm:main_theorem_k_var}]
Select $a_1, \dots, a_k$ to be coprime integers so that $a_1 = 1$, and $|a_i|$ is sufficiently large for $i \neq 1$. Let $S$ be a large set of integers.

With Lemma \ref{lem:easy_TNRG_lemma_k_var} replacing Lemma \ref{lem:easy_TNRG_lemma}, much of the argument is the same as the proof of Theorem \ref{thm:main_theorem}: 
\begin{itemize}
\item We start by using Theorem \ref{thm:green_sisask_structure} to split $S = S_1 \amalg \dots \amalg S_n \amalg S_0$, and show that $S_0$ can be ignored.
\item We can define the labelled digraph $G$ which captures almost all of the solutions to $a_1x_1 + \dots + a_kx_k = 0$.
\item We can prove that the product of the labels along a cycle must be 1, allowing us to define the function $d : [n] \rightarrow \Q$ as in (\ref{eqn:defn_of_depth}).
\item This allows us to show that an upper bound for the number of solutions coming from good $k$-tuples is 
\[
\sum_d T(R_{d/a_1}, \dots, R_{d/a_k}),
\]
where $R_{d}$ is the union of the $S_i$ with $d(i) = d$ (as in the case $k=3$, this sum can have no more than $n$ terms).
\end{itemize}

Lemma \ref{lem:uniform_bound_symmetric_k_var} allows us to bound this:
\begin{align}
\sum_d T(R_{d/a_1}, \dots, R_{d/a_k}) &\leq \sum_d \frac{\sigma_k}{k} \sum_i \widehat{R_{d/a_i}} + O(K_1 |S|^{k-2}) \nonumber\\
&\leq \frac{\sigma_k}{k} \sum_{(d_1, \dots, d_{k-1})} |R_{d_1}|\dots|R_{d_{k-1}}| + O(K_1 |S|^{k-2}).\label{eqn:sum_as_Rdi_k}
\end{align}
On the first line, $\widehat{R_{a/d_i}}$ denotes the product of the other $|R_{d/a_j}|$, and the error term comes from the fact that there are at most $n < K_1$ terms in the sum on the left hand side. On the second line, the sum is over unordered $k-1$-tuples $(d_1, \dots, d_{k-1})$ for which, for some ordering of the $a_i$, we have that $a_{i_1}d_1 = \dots = a_{i_{k-1}}d_{k-1}$; there can only be one such ordering by coprimality.

Now, for $i = 0, 1, \dots, k-1$, define the quantity $X_i$ by
\[
X_i = \sum_{\substack{d = a_2^{e_2} \dots a_k^{e_k} \\ e_2 + 2e_3 + \dots + (k-1)e_k \equiv d \mod k}} |R_d|.
\]
Note that if $d$ is such that $R_d$ is nonempty, then the representation of $d$ as a product $d = a_2^{e_2} \dots a_k^{e_k}$ exists due to how we constructed the labels, and is unique due to the coprimality of the $a_i$.

Now, suppose that $R_{d_1}$ and $R_{d_2}$ appear together in at least one term on the right hand side of (\ref{eqn:sum_as_Rdi_k}). Then, $d_1/d_2 = a_i/a_j$ for some $i \neq j$, and so $R_{d_1} \nd R_{d_2}$ contribute to different $X_i$.

Thus, we may upper bound the sum in the right hand side of (\ref{eqn:sum_as_Rdi_k}):
\begin{align*}
\sum_{(d_1, \dots, d_{k-1})} |R_{d_1}|\dots|R_{d_{k-1}}| \leq \sum_i \widehat{X_i},
\end{align*}
where $\widehat{X_i} = \prod_{j \neq i}X_j$. This bound follows from the fact that each unordered $k-1$-tuple $|R_{d_1}|\dots |R_{d_{k-1}}|$ on the left hand side contributes to exactly one of the terms on the right hand side.

Finally, observe that 
\[
\sum_i \widehat{X_i} \leq \frac{1}{k^{k-2}} \left(\sum_i X_i\right)^{k-1}.
\]
To see why, observe that if $X_i + X_j$ is kept fixed, moving $X_i$ and $X_j$ closer together increases the value of the left hand side without changing the right hand side. Thus the left hand side is maximised when the $X_i$ are all the same, at which point equality occurs.

Putting all of this together, we learn that 
\[
\sum_d T(R_{d/a_1}, \dots, R_{d/a_k}) \leq \frac{\sigma_k}{k^{k-1}} |S|^{k-1} + O(|S|^{k-2}),
\]
which gives the bound in the statement of Theorem \ref{thm:main_theorem_k_var} when $|S|$ is large enough.
\end{proof}

\section{Systems of more than one equation}\label{sec:generalisation_to_systems}

Another way in which one might wish to extend Theorem \ref{thm:main_theorem} is to ask if a similar result holds for systems of $m$ equations in $k$ variables. One might imagine that a result of the following form ought to hold.

\begin{question*}
Suppose that $k \geq m+2$ and $m \geq 1.$ Does there exist an explicit positive constant $\sigma_{m, k}$ with the following properties:

\begin{itemize}
\item For any system $\mathcal{A}$ of $m$ equations in $k$ variables, there are be large sets $S$, for which there are at least $(\sigma_{m, k} - o(1)) |S|^{k-m}$ $k$-tuples in $S$ satisfying $\mathcal{A}$.

\item For any $\eps > 0$, there are systems such that the number of $k$-tuples satisfying $\mathcal{A}$ in \emph{any} large $S \subseteq \Z$ is no more than $(\sigma_{m, k} + \eps)|S|^{k - m}$.
\end{itemize}
\end{question*}

Thus, Theorems \ref{thm:main_theorem} and \ref{thm:main_theorem_k_var} tell us that $\sigma_{m, k}$ exists whenever $m = 1$, and that $\sigma_{1, k} = \sigma_k$. However, it turns out that when $m > 1$, not even the first of these has a positive answer, in the following sense.
\begin{theorem}\label{thm:no_system_thm}
Let $\eps > 0$. Then, there exists a non-degenerate system of two equations in four variables with the property that for any large enough $S$, there are no more than $\eps |S|^2$ solutions to the system in $S$.
\end{theorem}

\begin{remark*}
It is easy to see that Theorem \ref{thm:no_system_thm} implies the analogous result for any choice of $k, m$ with $k \geq m+2$ and $m > 1$.
\end{remark*}

The goal of this section is to prove Theorem \ref{thm:no_system_thm}.
\begin{proof}
We will prove Theorem \ref{thm:no_system_thm} for the following system:
\begin{nalign}\label{eqn:the_system}
x + y &= z \\
x + My &= w,
\end{nalign}
where $M$ is a sufficiently large constant (in terms of $\eps$) to be chosen later.

We will start by borrowing the following lemma, which appears as part of the proof of the Balog-Szemer\'edi-Gowers theorem.

\begin{lemma}[\protect{\cite[Corollary 6.20]{TaoVu}}]\label{lem:BSGLemma}
Let $G$ be a bipartite graph with vertex sets $A$ and $B$ and edge set $E \subseteq A \times B$. Suppose that $|E| \geq \eps |A||B|$, for some $\eps > 0$. Then we can find subsets $A' \subseteq A$ and $B' \subseteq B$, with $|A'| \gg_\eps |A|$ and $|B'| \gg_\eps |B|,$ such that, whenever $a \in A'$ and $b \in B'$, there are $ \gg_\eps |A||B|$ paths of length three from $a$ to $b$ in $G$.
\end{lemma}

Let $S$ be a sufficiently large set (in terms of $M \nd \eps$), and suppose that there are more than $\eps |S|^2$ solutions to (\ref{eqn:the_system}) in $S$. Consider the bipartite graph on vertex set $A \amalg B$, where $A = B = S$; that is, both parts of $G$ are $S$. Draw an edge from $a$ to $b$ if and only if there is a solution to (\ref{eqn:the_system}) with $x = a \nd y = b$; in other words, if $a + b \nd a + Mb$ are both in $S$. In particular, $G$ has at least $\eps |S|^2$ edges.

We may immediately apply Lemma \ref{lem:BSGLemma} to $G$. This gives us sets $A' \subseteq A \nd B' \subseteq B$ such that, for any $a \in A' \nd b \in B'$, there are $\gg_\eps |S|^2$ paths of length 3 in $G$ from $a$ to $b$.

\begin{claim*}
These sets $A'$ and $B'$ satisfy $|A' + B'| \ll_\eps |S|$ and $|A' + M \cdot B'| \ll_\eps |S|$.
\end{claim*}

\begin{proof}[Proof of Claim]
To prove this claim, we can use an argument similar to that used in the proof of the Balog-Szemer\'edi-Gowers theorem. Showing that $|A' + B'| \ll_\eps |S|$ and $|A' + M \cdot B'| \ll_\eps |S|$ are similar, so we will only do the former.

Let $X$ denote the set of triples $(x, y, z)$ of elements of $(A + B) \cap S$, for which $x - y + z \in A' + B'$. We may trivially upper bound $|X|$; indeed, $|(A + B) \cap S| \leq |S|$, so $|X| \leq |S|^3$.

For a lower bound on $|X|$, consider an element $a + b$ of $A' + B'$. By definition, there are $\gg_\eps |S|^2$ paths of length 3 from $a$ to $b$ in $G$. Each such path may be written $a \sim b', a' \sim b', a' \sim b$ for some $a' \in A, b' \in B$. In other words, $a + b', a' + b' \nd a + b'$ are all in $S$.

Now, $(a+b') - (a'+b') + (a'+b) = (a+b),$ so we have located a triple $x, y, z \in (A+B)\cap S$ with $x - y + z = a+b$. These triples will be different for different paths, and so there must be $\gg_\eps |S|^2$ such triples.

There are $|A' + B'|$ elements of $A' + B'$, each of which gives $\gg_\eps |S|^2$ triples $x, y, z$. Thus, we have that $|A' + B'| |S|^2 \ll_\eps |S|^3$, and thus $|A' + B'| \ll_\eps|S|$, as required.\end{proof}

Let us now see how we may use this claim to complete the proof of Theorem \ref{thm:no_system_thm}. Lemma \ref{lem:ruzsa_triangle} immediately tells us that $|B' - M \cdot B'| \ll_\eps |S|$, and thus that $|B' - M \cdot B'| \ll_\eps |B'|$. This contradicts Theorem \ref{thm:bukh_sod_thm}, provided that $M$ is sufficiently large.

\end{proof}

\addtocontents{toc}{\protect\vspace*{\baselineskip}}
\nocite{HLP}


\addtocontents{toc}{\protect\vspace*{\baselineskip}}

\addcontentsline{toc}{section}{References}

\bibliographystyle{plainnat}

\begin{thebibliography}{15}
\providecommand{\natexlab}[1]{#1}
\providecommand{\url}[1]{\texttt{#1}}
\expandafter\ifx\csname urlstyle\endcsname\relax
  \providecommand{\doi}[1]{doi: #1}\else
  \providecommand{\doi}{doi: \begingroup \urlstyle{rm}\Url}\fi

\bibitem[Aaronson(2018)]{me}
J.~Aaronson.
\newblock Maximising the number of solutions to a linear equation in a set of
  integers.
\newblock \emph{arXiv preprint arXiv:1801.07135v2}, 2018.

\bibitem[Bhattacharya et~al.(2016)Bhattacharya, Ganguly, Shao, and
  Zhao]{uppertails}
B.~Bhattacharya, S.~Ganguly, X.~Shao, and Y.~Zhao.
\newblock Upper tails for arithmetic progressions in a random set.
\newblock \emph{arXiv preprint arXiv:1605.02994}, 2016.

\bibitem[Borwein and Borwein(2001)]{BorweinIntegral}
D.~Borwein and J.~Borwein.
\newblock Some remarkable properties of sinc and related integrals.
\newblock \emph{Ramanujan J.}, 5\penalty0 (1):\penalty0 73--89, 2001.

\bibitem[Bukh(2008)]{Bukh_sod}
B.~Bukh.
\newblock Sums of dilates.
\newblock \emph{Combin. Probab. Comput.}, 17\penalty0 (5):\penalty0 627--639,
  2008.

\bibitem[Gabriel(1931)]{Gabriel}
R.~M. Gabriel.
\newblock The {R}earrangement of {P}ositive {F}ourier {C}oefficients.
\newblock \emph{Proc. London Math. Soc. (2)}, 33\penalty0 (1):\penalty0 32--51,
  1931.

\bibitem[Ganguly(2018)]{pc_ganguly}
S.~Ganguly.
\newblock Personal communication, 2018.

\bibitem[Goddard(1945)]{GoddardAsymptotic}
L.~S. Goddard.
\newblock Lii. the accumulation of chance effects and the gaussian frequency
  distribution.
\newblock \emph{The London, Edinburgh, and Dublin Philosophical Magazine and
  Journal of Science}, 36\penalty0 (257):\penalty0 428--433, 1945.

\bibitem[Green and Sisask(2008)]{GreenSisask}
B.~Green and O.~Sisask.
\newblock On the maximal number of 3-term arithmetic progressions in subsets of
  {$\Bbb Z/p\Bbb Z$}.
\newblock \emph{Bull. Lond. Math. Soc.}, 40\penalty0 (6):\penalty0 945--955,
  2008.

\bibitem[Hardy and Littlewood(1928)]{HLIneqNotes}
G.~H. Hardy and J.~E. Littlewood.
\newblock Notes on the theory of series (viii): an inequality.
\newblock \emph{Journal of the London Mathematical Society}, 1\penalty0
  (2):\penalty0 105--110, 1928.

\bibitem[Hardy et~al.(1952)Hardy, Littlewood, and P{\'o}lya]{HLP}
G.~H. Hardy, J.~E. Littlewood, and G.~P{\'o}lya.
\newblock \emph{Inequalities}.
\newblock Cambridge university press, 1952.

\bibitem[Lev(1998)]{lev_max}
V.~Lev.
\newblock On the number of solutions of a linear equation over finite sets.
\newblock \emph{J. Combin. Theory Ser. A}, 83\penalty0 (2):\penalty0 251--267,
  1998.

\bibitem[Lev and Pinchasi(2014)]{LevPinchasi}
V.~Lev and R.~Pinchasi.
\newblock Solving {$a\pm b=2c$} in elements of finite sets.
\newblock \emph{Acta Arith.}, 163\penalty0 (2):\penalty0 127--140, 2014.

\bibitem[Nathanson(2014)]{nathanson}
M.~B. Nathanson.
\newblock Growth polynomials for additive quadruples and {$(h,k)$}-tuples.
\newblock \emph{Acta Math. Hungar.}, 143\penalty0 (1):\penalty0 44--57, 2014.

\bibitem[Polya(1913)]{Polya4Asymp}
G.~Polya.
\newblock Berechnung eines bestimmten {I}ntegrals.
\newblock \emph{Math. Ann.}, 74\penalty0 (2):\penalty0 204--212, 1913.

\bibitem[Tao and Vu(2010)]{TaoVu}
T.~Tao and V.~H. Vu.
\newblock \emph{Additive combinatorics}, volume 105 of \emph{Cambridge Studies
  in Advanced Mathematics}.
\newblock Cambridge University Press, Cambridge, 2010.

\end{thebibliography}

\end{document}